\documentclass[a4paper,12pt]{amsart}
\usepackage{amsmath, amssymb, amsthm}
\usepackage{url}
\usepackage{hyperref}
\usepackage{cleveref}
\usepackage{braket}
\theoremstyle{plain}

\newtheorem{thm}{Theorem}[section]
\newtheorem{theorem}[thm]{Theorem}
\newtheorem{lemma}[thm]{Lemma}
\newtheorem{corollary}[thm]{Corollary}
\newtheorem{cor}[thm]{Corollary}
\newtheorem{proposition}[thm]{Proposition}

\theoremstyle{remark}

\newtheorem{remark}[thm]{Remark}

\newtheorem{example}[thm]{Example}

\theoremstyle{definition}
\newtheorem{definition}[thm]{Definition}

\crefname{section}{Section}{Sections}

\crefname{thm}{Theorem}{Theorems}
\Crefname{thm}{Theorem}{Theorems}
\crefname{theorem}{Theorem}{Theorems}
\Crefname{theorem}{Theorem}{Theorems}
\crefname{lemma}{Lemma}{Lemmas}
\Crefname{lemma}{Lemma}{Lemmas}
\crefname{corollary}{Corollary}{Corollaries}
\Crefname{corollary}{Corollary}{Corollaries}
\crefname{cor}{Corollary}{Corollaries}
\Crefname{cor}{Corollary}{Corollaries}
\crefname{proposition}{Proposition}{Propositions}
\Crefname{proposition}{Proposition}{Propositions}
\crefname{prop}{Proposition}{Propositions}
\Crefname{prop}{Proposition}{Propositions}
\crefname{remark}{Remark}{Remarks}
\Crefname{remark}{Remark}{Remarks}
\crefname{rem}{Remark}{Remarks}
\Crefname{rem}{Remark}{Remarks}
\crefname{example}{Example}{Examples}
\Crefname{Example}{Example}{Examples}
\crefname{exercise}{Exercise}{Exercises}
\Crefname{exercise}{Exercise}{Exercises}
\crefname{definition}{Definition}{Definitions}
\Crefname{definition}{Definition}{Definitions}
\crefname{dfn}{Definition}{Definitions}
\Crefname{dfn}{Definition}{Definitions}
\crefname{algorithm}{Algorithm}{Algorithms}
\Crefname{algorithm}{Algorithm}{Algorithms}
\crefname{question}{Question}{Questions}
\Crefname{question}{Question}{Questions}
\crefname{problem}{Problem}{Problems}
\Crefname{problem}{Problem}{Problems}
\crefname{notation}{Notation}{Notations}
\Crefname{notation}{Notation}{Notations}
\crefname{conjecture}{Conjecture}{Conjectures}
\Crefname{conjecture}{Conjecture}{Conjectures}
\crefname{conj}{Conjecture}{Conjectures}
\Crefname{conj}{Conjecture}{Conjectures}
\crefname{condition}{Condition}{Conditions}
\Crefname{condition}{Condition}{Conditions}

\newcommand{\ZZ}{\mathbb{Z}}

\newcommand{\RR}{\mathbb{R}}

\newcommand{\ee}{\boldsymbol{e}}

\newcommand{\conv}{\operatorname{conv}}

\newcommand{\rank}{\operatorname{rank}}

\newcommand{\KR}{\operatorname{KR}}
\newcommand{\SE}{\operatorname{SE}}
\newcommand{\DE}{\operatorname{DE}}
\newcommand{\aff}{\operatorname{aff}}
\newcommand{\ver}{\operatorname{vert}}

\makeatletter
\DeclareSymbolFont{symbolsC}{U}{txsyc}{m}{n}
\DeclareMathSymbol{\MYPerp}{\mathrel}{symbolsC}{121}
\makeatother

\newcommand{\numof}[1]{\left|#1\right|}
\newcommand{\defit}[1]{\emph{#1}}

\newcommand{\ideal}[2]{%
\left(#1 \mathrel{}\middle|\mathrel{}  #2\right)
}

\allowdisplaybreaks[3] %denine 0-4 allow

\begin{document}

\title[Faces of Directed Edge Polytopes]{Faces of Directed Edge Polytopes}

\author[Y. Numata]{Yasuhide NUMATA}
\address[Y. Numata]{Department of Mathematics, Shinshu University, Matsumoto, Japan.}
\thanks{The first author was partially supported by JSPS KAKENHI Grant Number JP18K03206.}
% \curraddr{}
\email{nu@math.shinshu-u.ac.jp}
% \urladdr{}

\author[Y. Takahashi]{Yusuke TAKAHASHI}
\address[Y. Takahashi]{Graduate School of Science and Technology, Shinshu University, Matsumoto, Japan.}
% \thanks{}
% \curraddr{}
% \email{}
% \urladdr{}

\author[D. Tamaki]{Dai Tamaki}
\address[D. Tamaki]{Department of Mathematics, Shinshu University, Matsumoto, Japan.}
 \thanks{The third author was partially supported by JSPS KAKENHI Grant
 Number JP20K03579}
% \curraddr{}
\email{rivulus@shinshu-u.ac.jp}
% \urladdr{}

\keywords{Kantorovich--Rubinstein polytopes; fundamental polytopes; symmetric edge polytopes; root polytopes; characterization of faces; $f$-vectors}
\subjclass[2020]{52B05}

\begin{abstract}
 Given a finite quiver (directed graph) without loops and
 multiedges, the convex hull of the column vector of the incidence
 matrix is called the directed edge polytope and is an interesting
 example of lattice polytopes. 
 In this paper, we give a complete characterization of facets of the
 directed edge polytope of an arbitrary finite quiver without loops and
 multiedges in terms of the connectivity and the existence of a rank
 function. Our result can be regarded as an extension of the result of
 Higashitani et al.~\cite{MR3949939} on facets of symmetric edge
 polytopes to directed edge polytopes.
 When the quiver in question has a rank function, we obtain a
 characterization of faces of arbitrary dimensions.
\end{abstract}

\maketitle

% Introduction
% !TeX root =./x2.tex
% !TeX program = pdfLaTeX
\section{Introduction}

Motivated by optimal transportation problems, Vershik \cite{MR3331969}
proposed to study the convex polytope $\KR(X,d)$ constructed from a finite
metric space $(X,d)$. When $X=\{1,\ldots,n\}$, it is defined by
\[
 \KR(X,d) = \conv\ideal{\textstyle \frac{\ee_{i}-\ee_j}{d(i,j)}}{1\le
 i,j\le n},
\]
where $\{\ee_{1},\ldots,\ee_{n}\}$ is the standard orthonormal
basis of $\RR^{n}$.
This is called the \emph{fundamental polytope} in Vershik's paper. It is
also called the \emph{Kantorovich--Rubinstein polytope}
\cite{MR3810570,MR4249903}.
When $X$ is a tree-like metric space, Delucchi and Hoessley
\cite{MR4081480} proved a nice formula of the $f$-vector by using the
relation between tree-like metric spaces and hyperplane arrangements.
The starting point of this work was second author's attempt to extend
their work to graphs with cycles.

Given a finite simple graph $G$, define a metric $d_{\mathrm{graph}}$ on
the vertex set $G_{0}$ by the shortest length of paths, where each edge
is equipped with length one.
It turns out that the Kantorovich--Rubinstein polytope of
the metric space $(G_{0},d_{\mathrm{graph}})$ has already been studied
under different names.
It coincides with the \emph{symmetric edge polytope} $\SE(G)$ introduced
by Matsui et al.~\cite{MR2842918}.
When $G$ is the complete graph $K_{n}$, it is called the \emph{root
polytope} of the root system $A_{n}$ and its faces are completely
determined by Cho \cite{MR1697418}.

We may also generalize the definition of symmetric edge polytope to
a finite quiver (directed graph) $Q=(Q_{0},Q_{1})$ without loops and
multiedges by 
\[
 \DE(Q) = \conv\ideal{\varepsilon_{(i,j)}=\ee_{i}-\ee_{j}}{(i,j)\in Q_{1}}.
\]
Here $Q_{0}$ is the set of vertices and the set of edges $Q_{1}$ is
regarded as a subset of $Q_{0}\times Q_{0}$.
The polytope $\mathrm{DE}(Q)$ is called the \emph{directed edge
polytope} of $Q$ in \cite{MR3395979}.
The symmetric edge polytope $\SE(G)$ is nothing but $\DE(D(G))$,
where $D(G)$ is the double of $G$, i.e.\ the quiver obtained from $G$ by
replacing each edge $v - w$ by two directed edges $v\to w$ and
$v\leftarrow w$.

The aim of this paper is to find an explicit combinatorial description
of all facets of $\DE(Q)$ and thus obtain combinatorial descriptions of
facets of $\SE(G)=\KR(G_{0},d_{\mathrm{graph}})$ for a finite simple
graph. 

In general, the directed edge polytope $\DE(Q)$ is defined as a convex
polytope in the vector space $\RR^{Q_{0}}=\{\rho:Q_{0}\to\RR\}$.
Since the vertex set of $\DE(Q)$ is given by
$\Set{\varepsilon_{(v,w)}|(v,w)\in Q_{1}}$, any face of $\DE(Q)$ can be
written in the form $\DE(R)$ for a subquiver $R$ with
$R_{0}=Q_0$.  
Let us call such a subquiver a \emph{lluf subquiver}.

Given a lluf subquiver $R$ of $Q$, our problem is thus to
determine when $\DE(R)$ is a facet of $\DE(Q)$, i.e.\ 
$\DE(R)$ is a face of $\DE(Q)$ and
$\dim\DE(R)=\dim\DE(Q)-1$.  
It turns out that the existence of a \emph{rank function}, i.e.~a
function $\rho:Q_{0}\to\RR$ satisfying $\rho(v)+1=\rho(w)$ for any edge
$v\to w$, plays a key role in both problems.
As we see in \cref{prop:fuga}, such a function makes the
vertex set $Q_{0}$ into a graded poset. 

\begin{theorem}
\label{thm:dim}
 For a finite quiver $Q$ without loops and multiedges, we have
 \[
 \dim(\DE(Q))=
  \begin{cases}
   \numof{Q_{0}}-\numof{\pi_{0}(Q)}-1 & (\text{if $Q$ has a rank
   function}) \\ 
   \numof{Q_{0}}-\numof{\pi_{0}(Q)} & (\text{otherwise}),
\end{cases}
 \]
 where $\pi_{0}(Q)$ is the set of connected components of $Q$. 
\end{theorem}

%Let us denote $c(Q)=\numof{Q_{0}}-\numof{\pi_{0}(Q)}$ and call it the
%\emph{coconnectivity} of $Q$. 
It should be noted that, even if $Q$ is connected, a subquiver $R$
representing a facet of $\DE(Q)$ as $\DE(R)$ might not be connected.
In fact, the number of connected components is another key player in our
work. 

\begin{theorem}
 \label{thm:facet}
 Let $Q$ be a finite quiver without loops and multiedges.
 For a lluf subquiver $R$ of $Q$ with
 $\dim(\DE(R))=\dim(\DE(Q))-1$, 
 $\DE(R)$ is a facet of $\DE(Q)$ if and only if one of the
 following conditions holds:
 \begin{enumerate}
  \item \label{condition:hogera}
	$\numof{\pi_{0}(R)}=\numof{\pi_{0}(Q)}+1$, each connected
	component of $R$ is a full  
	subquiver of $Q$, and the contraction of $R$ in $Q$
	(\Cref{def:contraction})   
	$Q/R$ is acyclic.
  \item \label{condition:piyo} $\numof{\pi_{0}(R)}=\numof{\pi_{0}(Q)}$
	and there exists a rank function $\rho$ of $R$ such that
	\[
	 (\rho(v)-\rho(w)+1)(\rho(v')-\rho(w')+1)>0
	\]
	for any $(v,w),(v',w')\in Q_{1}\setminus R_{1}$.
 \end{enumerate} 
\end{theorem}

Note that lower dimensional faces can be obtained from facets by
iterating the process of taking facets.
Thus we obtain the following
characterization of all faces of $\DE(Q)$ for any quiver $Q$ with a rank
function.  

\begin{theorem}
 \label{thm:face}
 Suppose $Q$ has a rank function.
 For a proper subquiver $R$ of $Q$, the polytope $\DE(R)$ is a face of
 $\DE(Q)$, if and only if
 each connected component of $R$ is a full subquiver and $Q/R$ is
 acyclic. 
\end{theorem}

When $Q=D(G)$ for a simple graph $G$, it does not have a rank function
and the condition (\ref{condition:piyo}) in \cref{thm:facet} applies. It
is immediate to translate the condition (\ref{condition:piyo}) into the
following form. 

\begin{corollary}[\cref{cor:SE}]
 For a connected lluf subquiver $R$ of $D(G)$ with
 $\dim(\DE(R))=\dim(\SE(G))-1$, $\DE(R)$ is a facet of $\SE(G)$ if and
 only if
 $\numof{\pi_{0}(R)}=\numof{\pi_{0}(D(G))}$ and there exists a function
 $\rho\in \RR^{G_{0}}$ such that
 \[
 \rho(v)-\rho(w) =
 \begin{cases}
  1 & ((v,w)\in R_{1}) \\
  -1 & ((w,v)\in R_{1}) \\
  0 & (\text{otherwise})
 \end{cases}
 \]
 for $(v,w)\in D(G)_{1}$.
\end{corollary}

This is essentially equivalent to a characterization of facets of
symmetric edge polytopes in \cite{MR3949939} when $G$ is connected.

After fixing notation and terminology in \cref{sec:notation}, 
\cref{thm:dim} is proved in
\cref{sec:dim}, and \cref{thm:facet,thm:face} are proved in
\cref{sec:facet}.
We end this paper with sample computations in \cref{sec:examples}.
In particular, a complete characterization of all faces of the symmetric 
edge polytope of a cyclic graph is obtained, which was previously done
by the second author without using the characterization in this paper
and became the starting point of this 
work. 
%Note that the case of even cycles recently appeared in a paper
%\cite{1910.05193} by D'Ali, Delucchi, and Micha{\l}ek independently. 

%  Candidate of References:
%  \cite{MR1153934},
%  \cite{MR753872},
%  \cite{MR363963},
%  \cite{MR4081480},
%  \cite{MR4227793},
%  \cite{MR4245231},
%  \cite{MR1697418},
%  \cite{MR3331969},
%  \cite{MR357135},
%  \cite{MR2940501},
%  \cite{MR3949939},
%  \cite{2103.06404}

% Main part
%\input{p2.tex}
% !TeX root =./x2.tex
% !TeX program = pdfLaTeX
\def\Time-stamp: %
<#1 #2 #3>{\def\lastupdatetimestamp{#1 #2}}
\Time-stamp: <2022-03-01 09:41:09 nous>
%\def\thepage{\arabic{page} / ({\tiny \lastupdatetimestamp})}
%%% Time stamp macro for emacs.

\section{Notation and terminology}
\label{sec:notation}

First we fix notation and terminology for quivers.

\begin{definition}
 We call a pair $Q=(Q_0,Q_1)$
 a \defit{quiver}
 if $Q_0$ is a finite set and
 $Q_1$ is a subset of $Q_0\times Q_0 \setminus \Set{(v,v)|v\in Q_0}$.
 An element of $Q_0$ is called a \defit{vertex} of $Q$,
 and an element $(v,w)$ in $Q_1$ is called
 an \defit{edge} of $Q$ from $v$ to $w$.
\end{definition}

The following classes of quivers play essential roles in this paper.

\begin{definition}
 A quiver $Q$ is said to be
 \begin{enumerate}
  \item \defit{acyclic} if there does not exist $v_0,\ldots,v_n \in Q_{0}$
	such that
	$n>1$,  $v_n=v_0$, and 
	$(v_t,v_{t+1})\in Q_1$ for $t=0,\ldots,n-1$,

  \item \defit{asymmetric} if 
	\begin{align*}
	 (v,w)\in Q_1 \implies (w,v) \not\in Q_1,
	\end{align*}
	 and
  \item \defit{symmetric} if 
	\begin{align*}
	 (v,w)\in Q_1 \iff (w,v) \in Q_1.
	\end{align*}
 \end{enumerate}
\end{definition}

Note that a quiver may not be symmetric nor asymmetric.

\begin{definition}
% We call an undirected graph obtained in the following manner
% the \defit{underlying graph} of $Q$:
 We define the \defit{underlying graph} of a quiver $Q$
 to be
 the (undirected) graph obtained from $Q$ by using all vertices of $Q$ and
 by replacing all directed edges of $Q$ with undirected edges.
\end{definition}

Underlying graphs may have multiple edges.
The underlying graph of $Q$ is simple
if and only if $Q$ is asymmetric.

In order to describe faces of directed edge polytopes, we need subquivers.

\begin{definition}
 A quiver $R$ is called a \defit{subquiver} of $Q$
 if $R_0\subset Q_0$ and  $R_1\subset Q_1$.
% We write $R\subset Q$ to denote
% that $R$ is a subquiver of $Q$.
 We say that a subquiver $R$ of $Q$ is
 \begin{enumerate}
  \item \defit{proper}
	if $R_1$ is a proper subset of $Q_1$,
  \item \defit{full}
	if $R_1=\Set{(v,w)\in Q_1| v,w\in R_0}$,
	and
  \item \defit{lluf} if $R_0=Q_0$.
 \end{enumerate}
\end{definition} 

We make use of (undirected) walks to define connectivity of quivers.

\begin{definition}
 Let $Q$ be a quiver.
 An \defit{undirected walk} from $v_0$ to $v_n$ in $Q$
 is  a sequence $(v_0,v_1,\ldots,v_n)$ of vertices in $Q$
 such that $(v_t,v_{t+1})\in Q_1$ or $(v_{t+1},v_t)\in Q_1$ for all $t$.

 An undirected walk $(v_0,v_1,\ldots,v_n)$ is called
 \begin{enumerate}
%  \item an \defit{undirected path} if all vertices are distinct,
  \item \defit{closed} if $v_0=v_n$, and
  \item an \defit{undirected cycle} if it is closed and
	$v_i\neq v_j$ for any pair $(i,j)$ with $0\leq i<j<n$. 
 \end{enumerate}
\end{definition}

%In fact, the number of connected components is 

\begin{definition}
 We say a quiver $Q$ is \defit{connected} if, for any pair $(v,w)$ of
 vertices of $Q$, there exists an undirected walk from $v$ to $w$ .

 A connected maximal subquiver of $Q$
 is called a \defit{connected component} of $Q$. The set of all
 connected components of $Q$ is denoted by $\pi_{0}(Q)$. 
 The number $\numof{Q_{0}}-\numof{\pi_{0}(Q)}$ is denoted by $c(Q)$ and
 is called the \emph{coconnectivity} of $Q$.
\end{definition}

 Note that a quiver is connected if and only if the underlying graph is
 connected. 

\begin{definition}
 For a connected quiver $Q$, we call a lluf asymmetric subquiver 
 $R$ of $Q$ a \defit{spanning polytree} in $Q$ if the underlying graph
 of $R$ is a tree, 
 i.e.\ an acyclic connected simple undirected graph.
 For a quiver $Q$,
 we call a lluf subquiver $R$ of $Q$ a \defit{spanning polyforest} in $Q$
 if each connected component of $R$
 is a spanning polytree in some connected component of $Q$.
\end{definition}

We also need directed walks and cycles.

\begin{definition}
 A \defit{directed walk} from $v_0$ to $v_n$ in a quiver $Q$
 is a sequence $(v_0,v_1,\ldots,v_n)$ of vertices in $Q$
 such that $(v_t,v_{t+1})\in Q_1$ for all $t$.
 A directed walk $(v_0,v_1,\ldots,v_n)$ is called
 a \defit{directed cycle} if $v_{0}=v_{n}$ and
 $v_{i}\neq v_{j}$ for any pair $(i,j)$ with $0\leq i<j<n$. 
\end{definition}

By definition, a quiver is acyclic if and only if it does not contain a
directed cycle.

\section{Dimension}
\label{sec:dim}

Here is our main object of study.

\begin{definition}
Let $Q$ be a quiver.
The vector space of maps from $Q_0$ to $\RR$ is denoted by $\RR^{Q_0}$.
It is equipped with an inner product $\Braket{\ ,\ }$ defined by
\begin{align*}
 \Braket{\rho,\delta}=\sum_{v\in Q_0} \rho(v)\delta(v)
\end{align*}
for $\rho,\delta\in \RR^{Q_0}$.
For a subset $V\subset Q_0$,
we define an element $\kappa_{V} \in \RR^{Q_0}$ by
\begin{align*}
 \kappa_{V}(v)
=
\begin{cases}
1&(v\in V) \\
0&(v\not\in V)
 \end{cases}
\end{align*}
The set
 $\Set{\kappa_{\set{v}}|v\in Q_0}$ is
 a standard basis for the vector space $\RR^{Q_0}$. 
For $(v,w)\in Q_0\times Q_0$, 
we define the vector $\varepsilon_{(v,w)}$ by
\begin{align*}
\varepsilon_{(v,w)}= \kappa_{\set{v}} - \kappa_{\set{w}}.
\end{align*}
Define a convex polytope in $\RR^{Q_{0}}$ by
\begin{align*}
 \DE(Q)=\conv\Set{\varepsilon_{(v,w)}|(v,w)\in Q_1}.
\end{align*}
 This is called 
 the \defit{directed edge polytope} of $Q$.
% We define $\dim(\DE(Q))$ to be
% the dimension of the polytope $\DE(Q)$.
\end{definition}

\begin{remark}
 \label{remark:incidence_matrix}
 For a quiver $Q$, define a map
\[
 \delta: Q_{0}\times Q_{1} \longrightarrow \{-1,0,1\}
\]
 by 
 \[
 \delta(v,e) = 
 \begin{cases}
  1 & (e=(v,w) \text{ for some } w\in Q_{0}) \\
  -1 & (e=(w,v) \text{ for some } w\in Q_{0}) \\
  %% 1 & (e=(w,v) \text{ for some } w\in Q_{0}) \\
  %% -1 & (e=(v,w) \text{ for some } w\in Q_{0}) \\
  0 & (\text{ otherwise}).
 \end{cases}
 \]
 The matrix obtained from this map by choosing appropriate total orders
 of $Q_{0}$ and $Q_{1}$ is called the \emph{incidence matrix} of
 $Q$ and is denoted by $I(Q)$.
 Note that the vectors $\varepsilon_{(v,w)}$ are column vectors of
 $I(Q)$ and the directed edge polytope $\DE(Q)$ is the convex hull of these
 column vectors.
\end{remark}

It is easy to determine the vertices of $\DE(Q)$. For a convex polytope
$P$, let us denote the set of vertices of $P$ by $\ver(P)$.

\begin{lemma}
 \label{lem:vertices}
 For a quiver $Q$, we have
 \[
 \ver(\DE(Q)) = \Set{\varepsilon_{(v,w)}| (v,w)\in Q_{1}}.
 \]
\end{lemma}

In general, any face of a convex polytope is the convex hull of a
collection of vertices of the original polytope. In particular, any face
of $\DE(Q)$ is of the form $\DE(R)$ for a subquiver $R$.
The main purpose of this article is
to give a characterization of subquivers corresponding to faces.

Note that even if $Q$ is connected, a subquiver $R$ representing a face
of $\DE(Q)$ may not be connected.
It turns out that the number of connected components is closely related
to the dimension of $\DE(R)$.  
In fact, an upper bound is given by the
coconnectivity $c(R)$.

%For $\rho, \epsilon \in \RR^{Q_0}$,
%We define $\rho^{\perp}$ and $\rho^{\perp}+\epsilon$ by
%\begin{align*}
%\rho^{\perp}&=\Set{\delta\in\RR^{Q_0}| \Braket{\rho,\delta}=0},\\
%\rho^{\perp}+\epsilon&=\Set{\delta+\epsilon|\delta\in \rho^\perp}.
%\end{align*}
%First we discuss the upper bound of the dimension.
\begin{lemma}
 \label{lemma:ocs:constant}
 Define a vector subspace $V_{Q}$ of $\RR^{Q_{0}}$ by
 \[
  V_{Q} = \bigcap_{R\in\pi_{0}(Q)} \kappa_{R_{0}}^{\perp},
 \]
 where $\kappa_{R_{0}}^{\perp}$ is the orthogonal complement of
 $\kappa_{R_{0}}$ in $\RR^{Q_{0}}$.
 Then $\DE(Q)\subset V_{Q}$ and we have
 $\dim \DE(Q) \le \numof{Q_{0}}-\numof{\pi_{0}(Q)}=c(Q)$.
\end{lemma}

\begin{proof}
 Let $R$ be a connected component of $Q$.
 We show that $\DE(Q)$ is contained in the hyperplane
 $\kappa_{R_{0}}^{\perp}$. 
 For each edge $(v,w)$ in $Q$, either $\{v,w\}\subset R_{0}$ or
 $\{v,w\}\cap R_{0}=\emptyset$. In the former case,
 $\langle \kappa_{R_{0}},\varepsilon_{(v,w)}\rangle=1-1$ and in the
 latter case, 
 $\langle \kappa_{R_{0}},\varepsilon_{(v,w)}\rangle=0-0$.
 Thus $\DE(Q)\subset \kappa_{R_{0}}^{\perp}$.

 Since the hyperplanes defined by the vertex sets of connected
 components intersect transversally, we have
 \[
  \dim \DE(Q) \le \dim \left(\bigcap_{R\in\pi_{0}(Q)}
 \kappa_{R_{0}}^{\perp}\right) = \dim \RR^{Q_{0}} 
 - \numof{\pi_{0}(Q)} = c(Q). 
 \]
\end{proof}
%\begin{proof}
%  By definition,
%it follows that
%\begin{align*}
%  \Braket{\kappa_{\check Q_0},\varepsilon_{(v,w)}}
%  =
%\begin{cases}
%  1-1=0 & (v,w,\in \check Q_0)\\
%0-0=0 & (v,w,\not\in \check Q_0),
%\end{cases}
%\end{align*}
%which implies the lemma.
%\end{proof}
%We obtain the following 
%as a corollary to \cref{lemma:ocs:constant}
%\begin{cor}
%  \label{lemma:dimen:upperbound:general}
%  For a quiver $Q$, $\dim(\DE(Q))\leq\rank(Q)$.
%\end{cor}

It turns out that $\dim \DE(Q)$ varies
depending on the existence of a rank function, since such a function
defines another hyperplane that contains the directed edge polytope.

\begin{definition}
 For a quiver $Q$, a function $\rho\in\RR^{Q_{0}}$ is called a
 \emph{rank function} of $Q$ if it satisfies $\rho(v)+1=\rho(w)$ for
 each edge $(v,w)\in Q_{1}$.
\end{definition}

\begin{lemma}
 \label{lemma:ocs:rank}
 When $Q$ has a rank function $\rho$, choose an edge $(v,w)\in Q_{1}$
 and define a hyperplane $H_{\rho}$ in
 $\RR^{Q_{0}}$ by
 \[
  H_{\rho} = \rho^{\perp} + \varepsilon_{(v,w)} =
 \Set{\delta+\varepsilon_{(v,w)}|\delta\in \rho^\perp}.
 \]
 Then this is independent of the choice of an edge $(v,w)$ and contains
 $\DE(Q)$. 
\end{lemma}

\begin{proof}
 For edges $(v,w)$ and $(v',w')$ in $Q$, we have
 \[
 \langle \rho, \varepsilon_{(v,w)}\rangle = \rho(v)-\rho(w)  = -1 =
 \rho(v')-\rho(w') = 
 \langle \rho, \varepsilon_{(v',w')}\rangle,
 \]
 which implies that
 $\varepsilon_{(v,w)}-\varepsilon_{(v',w')}\in\rho^{\perp}$ and  
 $\rho^{\perp}+\varepsilon_{(v,w)}=\rho^{\perp}+\varepsilon_{(v',w')}$.
 It also implies that all vertices of $\DE(Q)$ are contained in
 $H_{\rho}$ and thus $\DE(Q)\subset H_{\rho}$.
\end{proof}

The hyperplane $H_{\rho}$ is transversal to the hyperplanes defined by
connected components of $Q$. And we have the following upper bound of
$\dim\DE(Q)$. 

\begin{corollary}
 \label{cor:dim:upperbound}
 If $Q$ has a rank function, then $\dim (\DE(Q))\le c(Q)-1$.
\end{corollary}

The choice of the term ``rank function'' is justified by the following
fact. 

%By definition,
%if $Q$ satisfies \Cref{condition:fuga} with $\rho$,
%then each subquiver of $Q$ also satisfies \Cref{condition:fuga}. 
%It is easy to show the following:
\begin{proposition}
 \label{prop:fuga}
 The following are equivalent for a quiver $Q$:
 \begin{enumerate}
  \item
       $Q$ has a rank function $\rho$.
  \item
       $Q$ is asymmetric and satisfies
       \begin{align*}
	\numof{\Set{t|(v_t,v_{t+1})\in Q_1}}
	=\numof{\Set{t|(v_{t+1},v_t)\in Q_1}}
       \end{align*}
       for each undirected closed walk
       $(v_0,v_1,\ldots,v_n)$ in $Q$.
\item
     $Q$ is asymmetric and satisfies
     \begin{align*}
      \numof{\Set{t|(v_t,v_{t+1})\in Q_1}}
      =\numof{\Set{t|(v_{t+1},v_t)\in Q_1}}
     \end{align*}
     for each undirected cycle
     $(v_0,v_1,\ldots,v_n)$ in $Q$.
  \item
       $Q$ is the Hasse diagram of a graded poset $(Q_0,\leq)$ with rank
       function $\rho$.
%       where a poset is said to be graded if there exists the rank function
%  $\rho\colon Q_0\to \ZZ$ such that $\rho(v)+1=\rho(w)$ for $v$ covered by $w$.
 \end{enumerate}
\end{proposition}

\begin{proof}
 Suppose $Q$ has a rank function $\rho$. Then $Q$ cannot have a pair
 $(v,w)$ of vertices with $v\to w$ and $v\leftarrow w$ and hence is
 asymmetric.  
 Let $(v_{0},\ldots,v_{n})$ be an undirected closed walk or cycle in $Q$.
 Then
 \[
  0 = \rho(v_{0})-\rho(v_{n}) = \sum_{t\in\Set{t|(v_{t},v_{t+1})\in
 Q_{1}}} 1 + \sum_{t\in\Set{t|(v_{t+1},v_{t})\in Q_{1}}}(-1)
 \]
 and we have
 \[
 \numof{\Set{t|(v_{t},v_{t+1})\in Q_{1}}} = \numof{\Set{t|(v_{t+1},v_{t})\in
 Q_{1}}}. 
 \]

 Conversely, suppose that the second condition is
 satisfied. Choose a vertex $v_{0}$. For a vertex $w\in Q_{0}$, choose
 an undirected walk $(v_{0},\ldots,v_{n}=w)$ from $v_{0}$ to $w$ and
 define
 \[
 \rho(w) = \numof{\Set{t|(v_{t},v_{t+1})\in Q_{1}}} -
 \numof{\Set{t|(v_{t+1},v_{t})\in Q_{1}}}. 
 \]
 The second condition guarantees that this is independent of the choice
 of a walk.

 Since any undirected closed walk can be decomposed into undirected
 cycles, the second and the third conditions are equivalent. 
 Finally the first and the fourth conditions are equivalent by definition.
\end{proof}

In order to obtain lower bounds of $\dim\DE(Q)$,
we consider the case of acyclic quivers.

%We define $\corank(Q)$ to be the number of weakly connected components of $Q$.
%We also define  $\rank(Q)$ to be $\rank(Q)=\numof{Q_0}-\corank(Q)$.

%\subsection{Main Theorems}
%First we discuss the dimension of a directed edge polytope.
%For a quiver $Q$,
%we define
%\cref{condition:fuga}
%as follows:
%\begin{condition}
%\label{condition:fuga}
%There exists $\rho\in\RR^{Q_0}$
%such that
%\begin{align*}
%  (v,w)\in Q_1 \implies \rho(v)+1=\rho(w).
%\end{align*}
%\end{condition}

%Next we discuss the lower bound of the dimension.
%\begin{lemma}
%\label{lem:dim:forest}
%If the underlying graph of an asymmetric quiver $Q$ is acyclic,
%then the dimension of the vector (resp.\ affine) space
%spanned by $\Set{\varepsilon_{(v,w)}|(v,w)\in Q_1}$ is $\rank(Q)$
%(resp.\ $\rank(Q)-1$).
%\end{lemma}

\begin{lemma}
 \label{lem:dim:forest}
 Let $F$ be a quiver whose underlying graph is acyclic.
 Then the dimension of the vector space spanned by
 $\set{\varepsilon_{(v,w)}|(v,w)\in F_{1}}$ is 
 given by $c(F)$.
 Thus we obtain
 \[
 \dim \DE(F) = \dim
 \aff\ideal{\varepsilon_{(v,w)}}{(v,w)\in F_{1}} = c(F) -1,  
 \]
 where $\aff$ denotes the affine hull.
\end{lemma}

\begin{proof}
 Recall that $I(F)$ is the incidence matrix of $F$. Then the dimension
 of the vector space spanned by 
 $\Set{\varepsilon_{(v,w)}|(v,w)\in F_{1}}$ is $\rank I(F)$.
 By the additivity of the rank of the incidence matrix with respect to
 disjoint unions, it suffices to prove that
 $\rank I(F)=\numof{F_{0}}-1$, when $F$ is connected.

 By the acyclicity assumption, the underlying graph of $F$ is a tree. 
 Let $n=\numof{F_{0}}$. We may choose an ordering
 $F_{0}=\{v_{1},\ldots,v_{n}\}$ in such a way that the underlying graph
 of the full subquiver $F^{(i)}$ with vertices $\{v_{i},\ldots,v_{n}\}$ is 
 a tree for each $i=1,\ldots,n-1$.
 In other words, $v_{i}$ is connected to a vertex in
 $\{v_{i+1},\ldots,v_{n}\}$ by a unique edge for each
 $i=1,\ldots,n-1$. 
 With this ordering, $I(F^{(i)})$ is of the form
 \[
  I(F^{(i)}) = 
 \begin{pmatrix}
  \pm 1 & 0 & \cdots & 0 \\
  0 & \\
  \vdots & \\
  0 & \\
  \mp 1 & & I(F^{(i+1)})& \\
 0 & \\
  \vdots & \\
  0 & 
 \end{pmatrix}
 \]
 and we obtain $\rank I(F)=n-1=\numof{F_{0}}-1$ by induction.

 Since the origin is not contained in the affine hull, we have
 \begin{align*}
 \dim\conv\ideal{\varepsilon_{(v,w)}}{(v,w)\in F_{1}} & = \dim
 \aff\ideal{\varepsilon_{(v,w)}}{(v,w)\in F_{1}} \\
  & = c(F)-1.  
 \end{align*}
\end{proof}
%\begin{proof}
%  For each weakly connected component $\check Q$ of $Q$,
%  we have the following:
%  Let $\numof{\check Q_0}=n$.
%  Since the underlying graph of $\check Q$ is a tree,
%  we can assume that $Q_0=\Set{v_1,\ldots,v_n}$
%  satisfies the following:
%  the underlying graph of the full subquiver of $\check Q$ with respect to 
%  $\Set{v_i,\ldots,v_n}$
%  is a tree for each $i=1,\ldots,n-1$.
%  For each $i=1,\ldots,n-1$,
%  there exists a unique edge $a_i$ connecting $v_i$ 
%  and a vertex in $\Set{v_{i+1},\ldots,v_n}$.
%  It is easy to show that
%  the rank of the incidence matrix of $\check Q$
%  with respect to these orderings
%  is full.
%  Hence the rank of the incidence matrix of $\check Q$
%  is $\rank(\check Q)$.
%
%  Since the rank of the incidence matrix of $\check Q$ is $\rank(\check Q)$
%  for each weakly connected component,
%  it follows that
%  the rank of the incidence matrix of $Q$ is $\rank(Q)$,
%  which implies the lemma.
%\end{proof}

%Since a quiver  $Q$
%has a spanning polyforest in $Q$,
%that is a weakly asymmetric subquiver whose underlying graph is acyclic,
%we obtain the following:
\begin{corollary}
\label{lem:dim:lowerbound}
 For any quiver $Q$, we have $\dim(\DE(Q))\geq c(Q)-1$.
\end{corollary}

\begin{proof}
  Choose a spanning polyforest $F$ in $Q$.
 Then we have
 \[
 \dim (\DE(Q)) \ge \dim (\DE(F)) = c(F)-1 \ge c(Q)-1,
 \]
 since $R_{0}=Q_{0}$.
%  If follows from \cref{lem:dim:forest}
%  that $\dim(\DE(F))\geq\rank(Q)-1$.
%  Since $\DE(F)\subset \DE(Q)$,
%  $\dim(\DE(Q))\geq\rank(Q)-1$.
\end{proof}

%Hence, e.g., an asymmetric quiver without undirected cycles
%has a rank function.

%We have the following equation for dimension of a direct edge polytope,
%which is proved in \cref{sec:proof}.
%\begin{theorem}
%  For a quiver $Q$,
%\begin{align*}
%\dim(\DE(Q))=
%  \begin{cases}
%\rank(Q)-1 &(\text{$Q$ satisfies \cref{condition:fuga}})\\
%\rank(Q)  &(\text{otherwise}).
%\end{cases}
%\end{align*}
%\end{theorem}
%

%
%\begin{lemma}
%\label{lemma:ocs:rank}
%Let $Q$ be a quiver.
%If $\rho \in \RR^{Q_0}$ satisfies
%$\rho(v)+1=\rho(w)$ for all $(v,w)\in Q_1$,
% then
%$\DE(Q)\subset \rho^\perp + \varepsilon_{(v,w)}$
%for $(v,w) \in \RR^{Q_0}$.
%\end{lemma}
%\begin{proof}
%  For each $(v,w)\in Q_1$,
%  it follows from direct calculation that
%  $\Braket{\rho,\varepsilon_{(v,w)}}=\rho(v)-\rho(w)=-1$.
%  Hence, for  $(v,w), (v',w')\in Q_1$,
%  we have $\Braket{\rho,\varepsilon_{(v',w')}-\varepsilon_{(v,w)}}=0$,
%  which implies
%  $\varepsilon_{(v',w')} \in \rho^\perp + \varepsilon_{(v,w)}$.
%  Since $\rho^\perp + \varepsilon_{(v,w)}$ contains
%  all vertices of $\DE(Q)$ 
%  the space also contains the convex hull $\DE(Q)$.
%\end{proof}

%We have the following as a corollary to 
%\cref{lemma:ocs:constant,lemma:ocs:rank}
%\begin{cor}
%  \label{cor:dim:upperbound}
%  If $Q$ satisfies \cref{condition:fuga},
%  then $\dim(\DE(Q))\leq\rank(Q)-1$.
%  If $Q$ does not satisfy \cref{condition:fuga},
%  then $\dim(\DE(Q))\leq\rank(Q)$.  
%\end{cor}

For those quivers that do not have rank functions, we have the following
lower bound.

\begin{lemma}
  \label{lem:dim:lowerbound1}
  If $Q$ does not have a rank function,
  then
  $\dim(\DE(Q))\geq c(Q)$.
\end{lemma}

\begin{proof}
 Let $F$ be spanning polyforest of $Q$. Then $F_{0}=Q_{0}$ and
 $\numof{\pi_{0}(F)} = \numof{\pi_{0}(Q)}$.
 By \cref{lem:dim:forest}, we have $\dim(\DE(F)) = c(F)-1=c(Q)-1$.

 Since the underlying graph of
 $F$ is acyclic, $F$ has a rank function $\rho$.
 Let $H_{\rho}$ be the hyperplane in \cref{lemma:ocs:rank}. It is
 given by
 \[
  H_{\rho} = \rho^{\perp} + \varepsilon_{(v_{0},w_{0})}
 \]
 for an edge $(v_{0},w_{0})$ in $F$.
 Since $Q$ does not have a rank function, there exists an edge $(v,w)$
 in $Q$ such that $\rho(v)+1\neq \rho(w)$. Then we have
 \begin{align*}
  \langle \rho, \varepsilon_{(v,w)}-\varepsilon_{(v_{0},w_{0})}\rangle & =
  \rho(v)-\rho(w) - (\rho(v_{0})-\rho(w_{0})) \\
  & = \rho(v)-\rho(w)+1 \neq 0,
 \end{align*}
 which implies that the vertex
 $\varepsilon_{(v,w)}$ is not contained in the hyperplane $H_{\rho}$.
 It is not contained in any one of hyperplanes of the form
 $\kappa_{G_{0}}$ for a connected component $G$
 of $F$, either. In other words,
 $\varepsilon_{(v,w)}\not\in H_{\rho}\cap V_{F}$.
 Since $ \DE(F) \subset H_{\rho} \cap V_{F}$, 
 we have $\dim(\DE(Q)) \ge \dim(\DE(F))+1 = c(Q)$.
\end{proof}
%\begin{proof}
%  Let $F$ be a spanning polyforest in $Q$.
%  Since the underlying graph of $F$ is acyclic,
%  we have $\rho\in \RR^{Q_0}$ satisfying
%  \begin{align*}
%    (v,w)\in F_1\implies \rho(v)+1=\rho(w).
%  \end{align*}
%  Let $(v_0,w_0)\in F_1$.
%  By \cref{lemma:ocs:constant,lemma:ocs:rank},
%  $\DE(F)$ is a $(\rank(Q)-1)$-dimensional polytope contained by
%  $\rho^\perp + \varepsilon_{(v_0,w_0)}$
%  and $\kappa_{\check Q_0}^\perp$ for all weakly connected component of $Q$.
%  Since $Q$ does not satisfy \cref{condition:fuga},
%  there exists $(v,w)\in Q_1$ such that
%  $\rho(v)+1\neq\rho(w)$.
%  Hence we have
%  \begin{align*}
%  \Braket{\rho,\varepsilon_{(v,w)}-\varepsilon_{(v_0,w_0)}}
%  &=(\rho(v)-\rho(w))-(\rho(v_0)-\rho(w_0))\\
%  &=\rho(v)-\rho(w)+1 \neq 0,
%  \end{align*}
%  which implies
%  $\rho^\perp + \varepsilon_{(v_0,w_0)}$
%  does not contain $\varepsilon_{(v,w)}$.
%  Let $S$ be 
%  the intersection of 
%  $\rho^\perp + \varepsilon_{(v_0,w_0)}$
%  and $\kappa_{\check Q_0}^\perp$ for all weakly connected component of $Q$.
%  Since $S$ 
%  is an affine space of dimension $\rank(Q)-1$ such that
%  $S$ contains $\DE(F)$ and $S$ does not contain $\DE(Q)$,
%  we have $\dim(\DE(Q))\geq\rank(Q)$.
%\end{proof}

Now \Cref{thm:dim} follows from 
\cref{cor:dim:upperbound}, \cref{lem:dim:lowerbound}, 
\cref{lemma:ocs:constant} and 
\cref{lem:dim:lowerbound1}.

\section{Facets}
\label{sec:facet}
%Here we consider the dimension and facets of a directed edge polytope.
%To describe dimensions, we define \cref{condition:fuga}, and show
%\cref{thm:dim}. 
%To characterize facets,
%we introduce \cref{condition:piyo,condition:hogera},
%and show \cref{thm:facet}.

Let $R$ be a lluf subquiver of $Q$ so that both $\DE(R)$
and $\DE(Q)$ are contained in $\RR^{Q_{0}}$. In order to prove
\cref{thm:facet}, we would like to know when
$\DE(R)$ is a face of $\DE(Q)$ and $\dim (\DE(R))=\dim(\DE(Q))-1$.

We first obtain the following relation between the coconnectivities of
$Q$ and $R$ by the dimension condition.

\begin{lemma}
 \label{lem:coconnectivity}
 Let $R$ be a lluf subquiver of $Q$. If $\DE(R)$ is a facet of $\DE(Q)$,
 then $c(R)=c(Q)$ or $c(R)=c(Q)-1$.
 Thus $\numof{\pi_{0}(R)}=\numof{\pi_{0}(Q)}$ or 
 $\numof{\pi_{0}(R)}=\numof{\pi_{0}(Q)}+1$. 
\end{lemma}

\begin{proof}
 When $Q$ has a rank function, so does $R$. And we have
 \[
 c(R)-1 = \dim(\DE(R)) = \dim(\DE(Q))-1 = c(Q)-2
 \]
 by \cref{thm:dim}, or $c(R)=c(Q)-1$.
 If $Q$ does not have a rank function, $\dim(\DE(Q))=c(Q)$, and we have
 $c(R)=c(Q)$ or $c(R)=c(Q)-1$, depending on the existence of a rank
 function on $R$.
\end{proof}

One of sufficient conditions for $\DE(R)$ being a facet is the
acyclicity of the quiver $Q/R$ obtained from $Q$ by ``contracting''
$R$. 

\begin{definition}
 \label{def:contraction}
 Let $R$ be a subquiver of a quiver $Q$ such that each
 connected component of $R$ is a full subquiver of
 $Q$. Define an equivalence relation $\sim$ on
 $Q_{0}$ by 
 \[
  v \sim w \Longleftrightarrow \text{$v$ and $w$
 are connected by an undirected walk in $R$}.
 \]
 The equivalence class of $v\in Q_{0}$ is denoted by $[v]$.
 Define a quiver $Q/R$ by 
 \begin{align*}
  (Q/R)_{0} & = Q_{0}/_{\sim} \\
  (Q/R)_{1} & = \Set{([v],[w]) | (v,w)\in Q_{1}\setminus R_{1}}. 
 \end{align*}
\end{definition}
%\begin{definition}
% \label{def:contraction}
%For a subquiver $\check Q$ of $Q$
%such that
%each connected component of $\check Q$ is a full subquiver of $Q$,
%we define a quiver $Q/\check Q$ to be the quiver $(Q_0/\check Q_0,Q_1/\check Q_1)$
%obtained by
%\begin{align*}
%  Q_0/\check Q_0 &= \Set{[v]|v\in Q_0},\\
%  Q_1/\check Q_1 &= \Set{([v],[w])|(v,w)\in Q_1\setminus \check Q_1}, 
%\end{align*}
%where
%$[v]=\Set{w|\text{$v$ and $w$ is weakly connected in $(Q_0,\check Q_1)$}}$.
%\end{definition}

Roughly speaking, $Q/R$ is the quiver obtained from $Q$ by
collapsing each connected component of $R$ to a point.

\begin{lemma}
 \label{lem:hogeraisface}
 Let $R$ be a lluf proper subquiver of a quiver $Q$ such that
 each connected component of $R$ is a full subquiver of
 $Q$. If $Q/R$ is acyclic, then $\DE(R)$ is a face of
 $\DE(Q)$. 
\end{lemma}

\begin{proof}
 Let us denote the connected components of $R$ by
 \[
  \pi_{0}(R) = \{R^{(1)},\ldots, R^{(n)}\}.
 \]
 Since $Q/R$ is acyclic, we may assume that, if
 $v\in R^{(i)}_{0}$ and $w\in R^{(j)}_{0}$ are connected
 by an edge in $Q$, then $i< j$.

 Denote $C_{k}= \bigcup_{i=1}^{k} R^{(i)}_{0}$.
 Then, for $v,w\in Q_{0}$,
 \begin{align*}
  \langle \kappa_{C_{k}},\varepsilon_{(v,w)}\rangle & =  
  \kappa_{C_{k}}(v) - \kappa_{C_{k}}(w) \\
  & = 
  \begin{cases}
   1 & (v\in C_{k} \text{ and } w\not\in C_{k}) \\
   -1 & (v\not\in C_{k} \text{ and } w\in C_{k}) \\
   0 & (\text{otherwise}).
  \end{cases}
 \end{align*}
 By our choice, the second case does not occur and we have
 $\langle\kappa_{C_{k}},\varepsilon_{(v,w)}\rangle\ge 0$.
 In other words, the orthogonal complement $\kappa_{C_{k}}^{\perp}$ is a
 supporting hyperplane of 
 $\DE(Q)$ and thus $\DE(Q)\cap\kappa_{C_{k}}^{\perp}$ is a face of
 $\DE(Q)$ for each $k$.

 We claim that
 \[
  \DE(Q)\cap \bigcap_{i=1}^{n} \kappa_{C_{k}}^{\perp} = \DE(R)
 \]
 or
 \[
 \Set{\varepsilon_{(v,w)}|(v,w)\in Q_{1}}
 \cap \bigcap_{i=1}^{n} \kappa_{C_{k}}^{\perp} =
 \Set{\varepsilon_{(v,w)}|(v,w)\in R}. 
 \]
 If $(v,w)\in R_{1}$, 
 $\langle \kappa_{C_{k}},\varepsilon_{(v,w)}\rangle =0$ for all
 $k$ by the previous calculation. Conversely, suppose that
 $(v,w)\in Q_{1}\setminus R_{1}$ with $v\in R^{(i)}_{0}$
 and $w\in R^{(j)}_{0}$. By assumption, $i<j$, which implies that
 \[
  \langle \kappa_{C_{i}},\varepsilon_{(v,w)}\rangle = 1-0=1\neq 0
 \]
 and we have $\varepsilon_{(v,w)}\not\in \kappa_{C_{i}}^{\perp}$.
\end{proof}

Another sufficient condition for being a face is the following.

\begin{lemma}
  \label{lem:piyoisface}
  Let $R$ be a lluf proper subquiver of a quiver $Q$.
  If $R$ has a rank function $\rho\in \RR^{Q_0}$ such that 
  \begin{align*}
    (\rho(v)-\rho(w)+1)(\rho(v')-\rho(w')+1)> 0
  \end{align*}
 for any $(v,w),(v',w')\in Q_{1}\setminus R_{1}$,
 then $\DE(R)$ is a face of $\DE(Q)$.
\end{lemma}
\begin{proof}
 Suppose a lluf subquiver $R$ of $Q$
 has such a rank function $\rho$.
 Fix $(v_0,w_0)\in R_1$ and consider the hyperplane
 \[
  H_{\rho} = \rho^{\perp}+\varepsilon_{(v_{0},w_{0})}.
 \]
 We claim that $H_{\rho}$ is a supporting hyperplane of $\DE(Q)$.
 
 For $(v,w)\in Q_1$, we have
 \begin{align*}
  \Braket{\rho, \varepsilon_{(v,w)}-\varepsilon_{(v_0,w_0)}} = 
  \rho(v)-\rho(w) - \rho(v_{0})+\rho(w_{0}) = \rho(v)-\rho(w)+1.
 \end{align*}
 Note that $\Braket{\rho, \varepsilon_{(v,w)}-\varepsilon_{(v_0,w_0)}}=0$
 for $(v,w)\in R_{1}$.
 By our assumption on $\rho$ we have
 $\Braket{\rho, \varepsilon_{(v,w)}-\varepsilon_{(v_0,w_0)}}\ge 0$
 for any $(v,w)\in Q_{1}$ or
 $\Braket{\rho, \varepsilon_{(v,w)}-\varepsilon_{(v_0,w_0)}}\le 0$
 for any $(v,w)\in Q_{1}$
 and $H_{\rho}$ is a supporting hyperplane of $\DE(Q)$.

 It remains to show that $\DE(Q)\cap H_{\rho}=\DE(R)$.
 Again by the assumption on $\rho$,
  \begin{align*}
    \Braket{\rho, \varepsilon_{(v,w)}-\varepsilon_{(v_0,w_0)}}=
    \rho(v)-\rho(w)+1=0 \iff  (v,w)\in R_{1}
  \end{align*}
  for $(v,w)\in Q_1$.
  Hence the hyperplane
  $H_{\rho}$ contains $\DE(R)$,
  but the hyperplane does not contain $ \varepsilon_{(v,w)}$
  for any $(v,w)\in Q_1\setminus R_1$.
  Hence the face $\DE(Q)\cap H_{\rho}$ coincides with $\DE(R)$.
\end{proof}

We next consider necessary conditions for being facets.

\begin{lemma}
  \label{lem:facetisfull}
  Let $R$ be a subquiver of $Q$
  with $\numof{\pi_{0}(R)}=\numof{\pi_{0}(Q)}+1$.
  If $\DE(R)$ is a facet of $\DE(Q)$,
  then each connected component of $R$
  is a full subquiver of $Q$.
\end{lemma}
\begin{proof}
 Let $n=\numof{\pi_{0}(R)}=\numof{\pi_{0}(Q)}+1$ and
 $R^{(1)},\ldots, R^{(n)}$ be the complete list of
 connected components of $R$ so that
 \[
  V_{R} = \bigcap_{k=1}^{n} \kappa_{R^{(k)}}^{\perp}.
 \]
 Let $(v,w)\in Q_1$ satisfy $v,w\in R^{(k)}_0$ for some $k$.
 It suffices to show that $\varepsilon_{(v,w)}\in \DE(R)$, since it is
 equivalent to $(v,w)\in R_{1}$ by \cref{lem:vertices}.

 By \cref{lemma:ocs:constant}, $V_{R}$ is a $c(R)$-dimensional affine
 space, which is a hyperplane in $V_{Q}$ by the assumption
 $c(R)=c(Q)-1$.  
 We have $\varepsilon_{(v,w)} \in V_{R}$, since
 \begin{align*}
  \kappa_{R^{(i)}_{0}}(\varepsilon_{(v,w)}) & =
  \begin{cases}
   1-1 & (i=k) \\
   0-0 & (i\neq k)
  \end{cases} \\
  & = 0
 \end{align*}
 for any $i=1,\ldots,n$.

% For all $v,w\in R^{(k)}_0$,
%  \begin{align*}
%    \braket{\kappa_{R^{(i)}_0},\varepsilon_{(v,w)}}=0.
%  \end{align*}
%  Let $A$ be the intersection of 
%  $\kappa_{R^{(k)}_0}^\perp$ for all $k$.
%  Then $A$ is an affine space of dimension $c(Q)$,
%  and $A$ contains $\varepsilon_{(v,w)}$.
  
  If $R$ does not have a rank function,
  then $\DE(R)$ is a $c(R)$-dimensional polytope
  contained in $V_{R}$. In other words, $V_{R}$ is a supporting
 hyperplane of $\DE(R)$ in $V_{Q}$, which implies that
  $\varepsilon_{(v,w)}\in \DE(R)$. 
  
 Suppose that $R$ has a rank function.
 $\DE(R)$ is of dimension $c(R)-1$ by \cref{thm:dim}.
 Since $\DE(R)$ is a facet of $\DE(Q)$,
 \[
  \dim(\DE(Q)) = \dim(\DE(R))+1 = c(R) = c(Q)-1.
 \]
 By \cref{thm:dim}, $Q$ also has a rank function, which is denoted by  
 $\rho$.
 It defines a
 hyperplane 
 \[
  H_{\rho} = \rho^{\perp} + \varepsilon_{(v_{0},w_{0})}
 \]
 in $\RR^{Q_{0}}$ for some $(v_{0},w_{0})\in Q_{1}$.
 We may choose $(v_{0},w_{0})\in R_{1}$.
 By \cref{lemma:ocs:rank}, $\DE(R)$ is contained in $H_{\rho}$, hence in
 $H_{\rho}\cap V_{R}$.
 
 Since $\rho$ is a rank function on $Q$, we have
 $\varepsilon_{(v,w)}\in H_{\rho}$ for $(v,w)\in Q_{1}$.
 If $v,w\in R_{1}^{(k)}$ for some $k$, we also have
 $\varepsilon_{(v,w)}\in V_{R}$, and hence
 $\varepsilon_{(v,w)} \in H_{\rho}\cap V_{R}$.
 Note that $H_{\rho}$ and $V_{R}$ intersect transversally, and we have 
 \[
  \dim (H_{\rho}\cap V_{R}) = \dim(V_{R})-1 = c(R)-1=\dim(\DE(R)).
 \]
 In other words, $H_{\rho}\cap V_{R}$ is the affine hull of $\DE(R)$ in
 $H_{\rho}$ and it should be the supporting hyperplane of $\DE(R)$ in
 $H_{\rho}$, since $\DE(R)$ is a facet of $\DE(Q)$.
 It implies that $\varepsilon_{(v,w)}\in \DE(R)$.
\end{proof}

\begin{lemma}
\label{lem:facetisacyclic}
  Let $R$ be a subquiver of $Q$
  with $\numof{\pi_{0}(R)}=\numof{\pi_{0}(Q)}+1$.
  If $\DE(R)$ is a facet of $\DE(Q)$,
  then $Q/R$ is acyclic.
\end{lemma}
\begin{proof}
% By assumption, $\numof{\pi_{0}(R)}=\numof{\pi_{0}(Q)}+1$.
 Denote
 \begin{align*}
  \pi_{0}(Q) & =\{Q^{(1)},\ldots, Q^{(n-1)}\} \\
  \pi_{0}(R) & =\{R^{(1)},\ldots, R^{(n)}\}.
 \end{align*}
 Without loss of generality, we may assume that $R^{(i)}\subset Q^{(i)}$
 for $i=1,\ldots,n-2$ and $R^{(n-1)}\cup R^{(n)}\subset Q^{(n-1)}$.
 Let us denote $R'=R^{(n-1)}\cup R^{(n)}$ and $Q'=Q^{(n-1)}$.
 We should have
 $R^{(i)}=Q^{(i)}$ for $i=1,\ldots,n-2$ and $\DE(R')$
 is a facet of $\DE(Q')$, since $\dim(\DE(R))=\dim(\DE(Q))-1$.
 It implies that $Q/R$ is a union of $(n-2)$
 quivers consisting of a single vertex and
 $Q'/R'$.

 If $Q/R$ were to have a directed cycle, it should be contained in
 $Q'/R'$. By \cref{lem:facetisfull},
 both $R^{(n-1)}$ and $R^{(n)}$ are full subquivers of $Q'$.
 Since $R$ is lluf, such a directed cycle contains edges
 $(v,w),(v',w')\in Q'_{1}$ with $v,w'\in R^{(n-1)}_{0}$ and
 $w,v'\in R^{(n)}_{0}$.
 Then we have
  \begin{align}
   \Braket{\kappa_{R^{(n-1)}_0},\varepsilon_{(v,w)}} &=1 \label{align:1} \\
    \Braket{\kappa_{R^{(n-1)}_0},\varepsilon_{(v',w')}}
   &=-1. \label{align:-1} 
  \end{align}
% and the hyperplane $\kappa_{R^{(n-1)}_{0}}^{\perp}$ is not a supporting
% hyperplane of $\DE(R')$. 

 If $R'$ does not have a rank function, this contracts to the fact that    
 $V_{R'}=\kappa_{R^{(n-1)}_{0}}^{\perp}\cap\kappa_{R^{(n)}_{0}}^{\perp}$
 is a supporting hyperplane of $\DE(R')$ in $V_{Q'}$,
 as we have seen in the proof of
 \cref{lem:facetisfull}.

 If $R'$ has a rank function, so does $Q'$ as is shown in the proof of
 \cref{lem:facetisfull}. Let $\rho$ be a rank 
 function of $Q'$. Then, again by the proof of \cref{lem:facetisfull},
 $H_{\rho}\cap V_{R'}$ is a supporting hyperplane of $\DE(R')$, which
 contradicts to (\ref{align:1}) and (\ref{align:-1}). 

 Hence $Q/R$ is acyclic.
\end{proof}  

\begin{lemma}
 \label{lem:facetispiyo}
 Let $R$ be a lluf subquiver of $Q$ with
 $\numof{\pi_{0}(R)}=\numof{\pi_{0}(Q)}$. 
 If $\DE(R)$ is a facet of $\DE(Q)$, then
 the subquiver $R$ has a rank function
 $\rho\in \RR^{Q_0}$ such that
 \[
 (\rho(v)-\rho(w)+1)(\rho(v')-\rho(w')+1)> 0.
  \]
 for $(v,w),(v' ,w') \in Q_1\setminus R_1$ 
 and $Q$ does not have a rank function.
\end{lemma}
\begin{proof}
 Since $c(R)=c(Q)$ and $\DE(R)$ is a facet of $\DE(Q)$,
 \[
  c(Q) \ge \dim(\DE(Q)) = \dim(\DE(R))+1 \ge c(R)=c(Q),
 \]
  $Q$ does not have a rank function and
  $R$ has a rank function by \cref{thm:dim}. We also have
 $\dim(\DE(Q))=c(R)$ and $\dim(\DE(R))=c(R)-1$.
  Since $\numof{\pi_{0}(Q)}=\numof{\pi_{0}(R)}$,
  $v$ and $w$ are in the same connected component of $R$
  for each $(v,w)\in Q_1$.
  Hence $V_{R}$ contains $\varepsilon_{(v,w)}$
  for all $(v,w)\in Q_1$. In other words, $\DE(Q)$ is a convex polytope
 in $V_{R}$.

 Let $\rho$ be a rank function of $R$.
 Fix $(v_0,w_0)\in R_1$ and consider the hyperplane
 $H_{\rho}=\rho^{\perp}+\varepsilon_{(v_{0},w_{0})}$.
 Then $H_{\rho}\cap V_{R}$ is an affine space of dimension $c(R)-1$
  which contains $\DE(R)$. It means that $H_{\rho}\cap V_{R}$ is a
 supporting hyperplane of $\DE(R)$ in $V_{R}$.
 Thus we have
    \begin{align*}
    \forall (v,w)\in Q_1,&\quad
    \Braket{\rho, \varepsilon_{(v,w)}-\varepsilon_{(v_0,w_0)}}=\rho(v)-\rho(w)+1\geq 0
    \intertext{or}
    \forall (v,w)\in Q_1,&\quad
    \Braket{\rho, \varepsilon_{(v,w)}-\varepsilon_{(v_0,w_0)}}=\rho(v)-\rho(w)+1\leq 0.
  \end{align*}

 It remains to show that these values are nonzero if 
 $(v,w)\in Q_{1}\setminus R_{1}$.
  If $\rho(v)-\rho(w)=1$,
  then $H_{\rho}\cap V_{R}$ contains $\varepsilon_{(v,w)}$.
  Since $\DE(R)$ is a facet  of $\DE(Q)$,
  it follows that
  $\DE(R)$ contains $\varepsilon_{(v,w)}$,
  which implies $(v,w)\in R_1$.
\end{proof}

Now we are ready to prove \cref{thm:facet,thm:face}.

\begin{proof}[Proof of \cref{thm:facet}]
 Let $R$ be a lluf subquiver of $Q$ with $\dim(\DE(R))=\dim(\DE(Q))-1$.
 
 Suppose that $\DE(R)$ is a facet of $\DE(Q)$. 
 By \cref{lem:coconnectivity}, $c(R)=c(Q)$  or $c(R)=c(Q)-1$.
 When $c(R)=c(Q)$, \cref{lem:facetispiyo} implies that the
 condition (\ref{condition:piyo}) in \cref{thm:facet} holds.
 If $c(R)=c(Q)-1$, the condition (\ref{condition:hogera}) in
 \cref{thm:facet} follows from  
 \cref{lem:facetisfull,lem:facetisacyclic}.  

 Conversely, if a lluf subquiver $R$ satisfies 
 (\ref{condition:hogera}) in \cref{thm:facet}, then $\DE(R)$ is a face
 of $\DE(Q)$ by \cref{lem:hogeraisface}. 
 If $R$ satisfies (\ref{condition:piyo}) in \cref{thm:facet}, $\DE(R)$
 is a face of $\DE(Q)$ by \cref{lem:piyoisface}. Hence, in both cases,
 $\DE(R)$ is a facet.  
\end{proof}

\begin{proof}[Proof of \cref{thm:face}]
 Suppose that $\DE(R)$ is a face of $\DE(Q)$. Then there exists a
 descending sequence of subquivers
%% \[
%%  R=Q^{(n)} \subset Q^{(n-1)} \subset \cdots \subset Q^{(1)}\subset
%%  Q^{(0)} =Q 
%% \]
 \[
 Q = Q^{(0)}
 \supset Q^{(1)}
 \supset \cdots
 \supset Q^{(n-1)}
 \supset  Q^{(n)}=R
\]
 such that $\DE(Q^{(i+1)})$ is a facet of $\DE(Q^{(i)})$ for each
 $i=0,\ldots,n-1$. 
 The rank function of $Q$ serves as a rank function of $Q^{(i)}$ and we have
 \[
  c(Q^{(i+1)})=\dim(\DE(Q^{(i+1)}))+1 = \dim(\DE(Q^{(i)}))-1+1=c(Q^{(i)})-1.
 \]
 by \cref{thm:dim}.
 It implies by \cref{thm:facet} that each connected component of
 $Q^{(i+1)}$ is a full subquiver of $Q^{(i)}$ and that
 $Q^{(i)}/Q^{(i+1)}$ is an acyclic quiver for all $i$.
 Hence
 each connected component of $Q^{(n)}=R$ is a full subquiver of $Q$
 and $Q^{(0)}/Q^{(n)}=Q/R$ is acyclic.

 The converse follows from \cref{lem:hogeraisface} .
\end{proof}

%% Assume that $\dim(\check P)+1$
%% Let $P(R)$ be a facet of $P(Q)$.
%% We have $\dim(P(R))=\dim(P(Q))-1$.
%% Since $\dim(P(R))\leq \rank(R)$ and $\rank(Q)-1\leq \dim(P(Q))$,
%% $\rank(R)\geq \rank(Q)-2$.
%% In the case where $\rank(R)= \rank(Q)-2$,
%% $\dim(P(R))=\dim(P(Q))-1$.
%% only if $\dim(P(R))=\rank(R)$ and $\dim(P(Q))=\rank(Q)-1$.
%% If $\dim(P(Q))=\rank(Q)-1$,
%% then $Q$ satisfies \cref{condition:fuga}.
%% Since $Q$ satisfies \cref{condition:fuga},
%% the subquiver $R$ of $Q$ also satisfies \cref{condition:fuga}.
%% Hence we have $\dim(P(R))=\rank(R)-1$.
%% Thus,
%% if  $P(R)$ is a facet of $P(Q)$,
%% then $\rank(R)$ is $\rank(Q)-1$ or $\rank(Q)$.
%% Therefore
%% \cref{thm:facet}
%% follows from
%% \cref{lem:piyoisface,lem:hogeraisface,lem:facetisacyclic,lem:facetisfull,lem:facetispiyo}.

\section{Examples}
\label{sec:examples}

Here we consider some special cases as applications of our results.
First we consider the case of an asymmetric quiver $Q$ with no
undirected closed walk. Namely
the underlying graph of $Q$ is a forest.
In this case, we have
\begin{align*}
 \numof{\ver(\DE(Q))} & = \numof{Q_1} \\
 & =\numof{Q_{0}}-\numof{\pi_{0}(Q)} \\
 & = c(Q). 
\end{align*}
On the other hand, $Q$ has a rank function and we have
$\dim(\DE(Q))=c(Q)-1$ by \cref{thm:dim}. 
Hence we have the following:
\begin{cor}
 \label{cor:acyclicgraph}
 Let $Q$ be an asymmetric quiver with $Q_1\neq \emptyset$.
  If the underlying graph of $Q$ is acyclic,
  then 
  $\DE(Q)$ is a simplex of dimension $\numof{Q_1}-1=c(Q)-1$.
\end{cor}

\begin{example}
 Let $Q$ be an asymmetric quiver whose underlying graph is the Dynkin
 graph $A_{n+1}$.
 Denote $Q_0=\Set{0,1,\ldots,n}$ and
 $Q_1=\Set{e_1,\ldots,e_n}$ so that $e_{i}=(i-1,i)$ or $(i,i-1)$ for
 $i=1,\ldots,n$. 
 Then the directed edge polytope $\DE(Q)$ is
 an $(n-1)$-simplex.
 Hence $\DE(R)$ is a face for any luff proper subquiver $R$
 of $Q$.
\end{example}

One of the simplest cases in which the underlying graph is not a tree is
the following. 

\begin{example}
 \label{example:asymmetriccycle}
 Let $Q$ be an asymmetric quiver whose underlying graph is
 the boundary of a $2n$-gon. Denote
 $Q_0=\ZZ/2n\ZZ=\Set{\overline{1},\ldots,\overline{2n}=\overline{0}}$
 and $Q_{1}=\Set{e_{1},\ldots,e_{2n}}$ so that
 $e_{i}=(\overline{i-1},\overline{i})$ or
 $(\overline{i},\overline{i-1})$ for each $i\in\Set{1,\ldots,2n}$.

 Define $Q^{+}_1$ and $Q^{-}_1$ by 
 \begin{align*}
  Q^{+}_1&=Q_1\cap \Set{(\overline{i-1},\overline{i})|\overline{i}\in Q_0},\\
  Q^{-}_1&=Q_1\cap \Set{(\overline{i},\overline{i-1})|\overline{i}\in Q_0}.
 \end{align*}
 If $\numof{Q^{+}_1}=\numof{Q^{-}_1}=n$, 
 then $Q$ has a rank function.
 By \cref{thm:dim},
 \[
  \dim(\DE(Q)) = c(Q)-1 = \numof{Q_{0}}-\numof{\pi_{0}(Q)} -1 = 2n-2.
 \]
 Since $\numof{\ver(\DE(Q))}=\numof{Q_{1}}=2n$, $\DE(Q)$ is not a
 simplex. 

 Let $R$ be a lluf subquiver of $Q$ whose directed edge polytope
 $\DE(R)$ is a facet of $\DE(Q)$. Since $R$ also has a rank function,
 \[
 2n-3=\dim(\DE(R)) = \numof{Q_{0}}-\numof{\pi_{0}(R)}-1
 =2n-1-\numof{\pi_{0}(R)}, 
 \]
 which implies that $Q_{1}\setminus R_{1}$ consists of two disjoint
 edges. Let $Q_1\setminus R_1 = \Set{e',e''}$. 
 The acyclicity of $Q/R$ following from \cref{thm:face} allows us to
 assume that $e'\in Q^{+}_{1}$ and $e''\in Q^{-}_{1}$.
 This is a characterization of facets of $\DE(Q)$.
 
 For such a subquiver $R$, $\DE(R)$ is a simplex
 of dimension $2n-3$ by \cref{cor:acyclicgraph}.
 Since faces of a simplex are in one-to-one correspondence to subsets of
 the vertex set, for a lluf subquiver $S$ of $Q$,
  $\DE(S)$ is a proper face of $\DE(Q)$ of dimension $d$
 if and only if $\numof{S_1\cap Q^{+}_1}<n$,
 $\numof{S_1\cap Q^{-}_1}<n$, and $\numof{S_1}=d+1$.
 Hence the number $f_d$ of faces of $\DE(Q)$ of dimension $d$ is given
 by 
 \begin{align*}
  f_d=\binom{2n}{d+1}-2\binom{n}{d+1-n},
 \end{align*}
  where the binomial coefficient $\binom{m}{k}$ equals $0$ if $m<k$ or
 $k<0$. 
\end{example}

Finally we consider the case of symmetric edge polytopes.
For a finite graph $G$, the symmetric edge polytope $\SE(G)$ of $G$
introduced by Matsui et al.~\cite{MR2842918} is, by definition, the
directed edge polytope $\DE(D(G))$ of the double $D(G)$ of $G$.
Note that any symmetric quiver in our sense is of the form $D(G)$ for a
finite graph $G$.
 
Since $D(G)$ is a symmetric quiver, $D(G)/R$ is not acyclic for any
proper subquiver $R$ of $D(G)$.
Hence we have the following:
\begin{cor}
 \label{cor:SE}
 Let $G$ be a finite simple graph whose vertex set is denoted by $G_{0}$.
 For a lluf subquiver $R$ of $D(G)$ with $\dim(\DE(R))=\dim(\SE(G))-1$,
 the following are equivalent:
 \begin{enumerate}
  \item $\DE(R)$ is a facet of $\SE(G)$.
  \item $c(R)=c(D(G))$ and there
	exists a rank function $\rho$ of $R$ such that
	\[
	(\rho(v)-\rho(w)+1)(\rho(v')-\rho(w')+1)>0
	\]
	for any $(v,w),(v',w')\in D(G)_{1}\setminus R_{1}$.
  \item \label{SE3}
	$c(R)=c(D(G))$ and there exists a function
	$\rho\in \RR^{G_{0}}$ such that
	\[
	\rho(v)-\rho(w) =
	\begin{cases}
	 1 & ((v,w)\in R_{1}) \\
	 -1 & ((w,v)\in R_{1}) \\
	 0 & (\text{otherwise})
	\end{cases}
	\]
	for $(v,w)\in D(G)_{1}$.
 \end{enumerate}
\end{cor}

 \begin{proof}
  The equivalence of first two conditions follows from
  \cref{thm:facet}. The third condition is easily seen to be equivalent
  to the second condition.
 \end{proof}

\begin{remark}
 In \cite{MR3949939},
 Higashitani, Jochemko and Mateusz
 obtained a characterization of facets of symmetric edge
 polytopes for a connected simple graph $G$ as the existence of a
 function $\rho: G_{0}\to\ZZ$ satisfying the following two conditions:
 \begin{enumerate}
  \item $\rho(v)-\rho(w)\in \Set{-1,0,1}$ for any edge $(v,w)\in D(G)$, and
  \item the underlying graph of the quiver $E^{\rho}$ defined by
	\[
	 E^{\rho}_{1} = \Set{(v,w)\in \SE(G)_{1} | \rho(v)=\rho(w)+1}
	\]
	is a spanning subgraph of $G$.
 \end{enumerate}
 
 Their characterization can be obtained from \cref{cor:SE} as follows. 
 Let $\rho\in\RR^{G_{0}}$ be a function satisfying the third condition
 of \cref{cor:SE}. 
 Then $\rho(v)-\rho(w)\in \Set{-1,0,1}$ for all $(v,w)\in Q_1$ and we
 may assume that $\rho$ takes values in $\ZZ$.
 For such a function $\rho$, the quiver $E^{\rho}$ can be easily seen to
 coincide with our
 subquiver $R$ in \cref{cor:SE}.
 Since $G$ is connected, so is $\SE(G)$. The condition $c(R)=c(D(G))$
 implies that $R$ is also connected. Since $R$ is a lluf subquiver, it
 implies that the underlying graph of $E^{\rho}=R$ is a spanning
 subgraph of $G$.
  
 Conversely, let $\rho:G_{0}\to\ZZ$ be a function satisfying
 $\rho(v)-\rho(w)\in \Set{-1,0,1}$ for all $(v,w)\in Q_1$ and suppose
 that the underlying graph of $E^{\rho}$ is a spanning subgraph of $G$.
 In particular, $E^{\rho}$ is a lluf subquiver of $D(G)$.
 By the connectivity of $E^{\rho}$, we see $c(E^{\rho})=c(D(G))$.
 Since $\rho$ is a rank function on $E^{\rho}$, it satisfies the
 condition (a) of the third condition of \cref{cor:SE}.
 The condition that $\rho(v)-\rho(w)\in\Set{-1,0,1}$ for
 $(v,w)\in D(G)$, then, implies the condition (b).
\end{remark}

The following two examples were first studied by the second author in an 
elementary method analogous to that of Cho's \cite{MR1697418}, whose
analysis led to the current work.

\begin{example}
 \label{example:evencycle}
 Let $C_{2n}$ be the boundary of $2n$-gon regarded as a graph of $2n$
 vertices and edges. As is the case of \cref{example:asymmetriccycle},
 the vertex set is identified with
 $\ZZ/2n\ZZ=\Set{\overline{1},\ldots,\overline{2n}=\overline{0}}$. 
 The symmetric edge polytope $\SE(C_{2n})=\DE(D(C_{2n}))$ is a
 $(2n-1)$-dimensional polytope by \cref{thm:dim}.
 The faces of $\SE(C_{2n})$ can be determined as follows.
 
 Denote $Q=D(C_{2n})$ for simplicity and define
 \begin{align*}
  Q^{+}_1&=\Set{(\overline{i-1},\overline{i})| i=1,\ldots,2n},\\
  Q^{-}_1&=\Set{(\overline{i},\overline{i-1})| i=1,\ldots,2n}
 \end{align*}
 so that $Q_{1}=Q_{1}^{+}\cup Q_{1}^{-}$.

 By (\ref{SE3}) of \cref{cor:SE}, for a lluf subquiver $R$ of
 $Q$ with
 \[
 \dim(\DE(R))=\dim(\SE(C_{2n}))-1=2n-2,  
 \]
 $\DE(R)$ is a facet
 of $\SE(C_{2n})$ if and only if 
 $c(R)=c(Q)=2n-1$ and
 there exists a function $\rho:\ZZ/2n\ZZ\to\RR$ such that
 \[
  \rho(\overline{i-1})-\rho(\overline{i}) =
 \begin{cases}
  1 & ((\overline{i-1},\overline{i})\in R_{1}) \\
  -1 & ((\overline{i},\overline{i-1})\in R_{1}) \\
  0 & (\text{otherwise}),
 \end{cases}
 \]
 which implies that only one of $(\overline{i-1},\overline{i})$ or
 $(\overline{i},\overline{i-1})$ belongs to $R_{1}$ for each $i$.
 Denote
 \begin{align*}
  R^{+}_{1} & = R_{1}\cap Q^{+}_{1} =
  \Set{(\overline{i-1},\overline{i}) | i\in I_{+}} \\ 
  R^{-}_{1} & = R_{1}\cap Q^{-}_{1} =
  \Set{(\overline{i},\overline{i-1}) | i\in I_{-}}. 
 \end{align*}
 Then
 \begin{align*}
  0 & = \sum_{i=1}^{2n} (\rho(\overline{i-1}) -\rho(\overline{i})) \\
  & = \sum_{i\in I_{+}} 1 + \sum_{i\in I_{-}}(-1) + \sum_{i\not\in
  I_{+}\cup I_{-}} 0 \\
  & = \numof{I_{+}} - \numof{I_{-}}
 \end{align*}
 and we have $\numof{I_{+}}=\numof{I_{-}}$.
 
 Since $\DE(R)$ is of dimension $2n-2$,
 \[
 \numof{R_{1}} = \numof{\ver(\DE(R))} \ge 2n-1.
 \]
 By the condition on $\rho$, we see that the underlying graph of $R$
 must be the whole $C_{2n}$.
 Thus we have $\numof{R^{+}_1}=\numof{R^{-}_1}=n$ and 
 $R^{+}_{1}\cap (-R^{-}_{1})=\emptyset$, where
 $-R^{-}_{1}=\Set{(\overline{i-1},\overline{i})|(\overline{i},\overline{i-1})\in
 R^{-}_{1}}$.
 In other words, facets of $\SE(C_{2n})$ are in bijective correspondence
 to subsets of cardinality $n$ in
 $Q_{1}=\Set{(\overline{i-1},\overline{i}),(\overline{i},\overline{i-1})
 | i=1,\ldots,2n}$.
 Hence the number $f_{2n-2}$ of facets of $\SE(C_{2n})$
 is given by
  \begin{align*}
    f_{2n-2} = \binom{2n}{n}.
  \end{align*}
 
 Note that facets of $\SE(C_{2n})$ are polytopes in
 \cref{example:asymmetriccycle}. 
 In particular, faces of codimension $2$ in $\SE(C_{2n})$ are simplices
 of dimension $(2n-3)$,
 which means that all faces of $\SE(C_{2n})$ except for facets are
 simplices. 
 In other words, for $d<2n-2$ and a lluf subquiver $R$ of $Q$,
 $\DE(R)$ is a face of dimension $d$ in $\SE(C_{2n})$
 if and only if $\numof{R_1}=d+1$,
 $\numof{R_1\cap Q^{+}_1}<n$,
 $\numof{R_1\cap Q^{-}_1}<n$, and $R^{+}_{1}\cap (-R^{-}_{1})=\emptyset$.
 Hence the number $f_d$ of faces of $\DE(Q)$ of dimension $d$ is given by
  \begin{align*}
    f_d&=\sum_{i\in I}\binom{2n}{i}\binom{2n-i}{d+1-i}\\
    &=\binom{2n}{d+1}\sum_{i\in I}\binom{d+1}{i},
  \end{align*}
  where $I=\Set{i\in \ZZ | i<n, d+1-i<n}$.
  If $d+1<n$,
  then we have $\sum_{i\in I}\binom{d+1}{i}=2^{d+1}$,
  which implies
  \begin{align*}
    f_d &=\binom{2n}{d+1}2^{d+1}.
  \end{align*}

 We remark that D'Ali, Delucchi, and Micha{\l}ek \cite{1910.05193} also
 performed the same computation based on the characterization of facets
 by Higashitani et al.~\cite{MR3949939}.
\end{example}

\begin{example}
 \label{example:oddcycle}
 Consider the case of an odd cycle $C_{2n+1}$. As is the case of
 \cref{example:evencycle}, we identify the vertex set with
 $\ZZ/(2n+1)\ZZ=\Set{\overline{1},\ldots,\overline{2n},\overline{2n+1}=0}$.
 For simplicity, we denote $Q=D(C_{2n+1})$ and 
 \begin{align*}
  Q^{+}_1&=\Set{(\overline{i-1},\overline{i})| i=1,\ldots,2n+1},\\
  Q^{-}_1&=\Set{(\overline{i},\overline{i-1})| i=1,\ldots,2n+1}.
 \end{align*}
 The symmetric edge polytope $\SE(C_{2n+1})=\DE(Q)$ is a
 $2n$-dimensional polytope by \cref{thm:dim}. 
 
 For a lluf subquiver $R$ of $Q$, suppose that $\dim(\DE(R))=2n-1$.
 By the same argument as in \cref{example:evencycle},
 $\DE(R)$ is a facet of $\SE(C_{2n+1})$ if and only if
 $c(R)=c(Q)$, 
 $\numof{R^{+}_1}=\numof{R^{-}_1}=n$,
 and $R^{+}_1\cap (-R^{-}_{1})=\emptyset$.
 Since $\numof{\ver(\DE(R))}=2n$, $\DE(R)$ is a simplex of dimension
 $2n-1$ and all faces of $\SE(C_{2n+1})$ are simplices.
 We also see that facets are in one-to-one correspondence to a pair
 $(E,e)$ of a subset $E$ of $Q^{+}_{1}$ of cardinality $n$ and an
 element $e\in Q^{-}_{1}\setminus(-E)$ and the number $f_{2n-1}$ of
 facets of $\DE(Q)$ 
  is given by
  \begin{align*}
    f_{2n-1} = (n+1)\binom{2n+1}{n} = \frac{(2n+1)!}{n!n!}=
   (2n+1)\binom{2n}{n}. 
  \end{align*}
 
 Thus, for $d<2n-2$ and a lluf subquiver $R$ of $Q$,
 $\DE(R)$ is a face of $\SE(C_{2n+1})$ of dimension $d$
 if and only if $\numof{R_1}=d+1$,
 $\numof{R^{+}_1}<n$, $\numof{R^{-}_1}<n$, and 
 $R^{+}_1\cap (-R^{-}_{1})=\emptyset$.
 Hence the number $f_d$ of faces of dimension $d$ in $\SE(C_{2n+1})$ is
 given by 
  \begin{align*}
    f_d&=\sum_{i\in I}\binom{2n+1}{i}\binom{2n+1-i}{d+1-i}\\
    &=\binom{2n+1}{d+1}\sum_{i\in I}\binom{d+1}{i},
  \end{align*}
  where $I=\Set{i\in \ZZ | i<n, d+1-i<n}$.
  If $d+1<n$,
  then we have $\sum_{i\in I}\binom{d+1}{i}=2^{d+1}$,
  which implies
  \begin{align*}
    f_d &=\binom{2n+1}{d+1}2^{d+1}.
  \end{align*}
\end{example}

% References
\bibliography{by-mr,by-arxiv}

\providecommand{\bysame}{\leavevmode\hbox to3em{\hrulefill}\thinspace}
\providecommand{\MR}{\relax\ifhmode\unskip\space\fi MR }
\providecommand{\MRhref}[2]{%
  \href{http://www.ams.org/mathscinet-getitem?mr=#1}{#2}
}
\providecommand{\href}[2]{#2}
\begin{thebibliography}{1}

\bibitem{MR1697418}
Soojin Cho, \emph{Polytopes of roots of type {$A_n$}}, Bull. Austral. Math.
  Soc. \textbf{59} (1999), no.~3, 391--402, URL
  \url{https://doi.org/10.1017/S0004972700033062}. \MR{1697418}

\bibitem{1910.05193}
Alessio D'Alì, Emanuele Delucchi, and Mateusz Micha{\l}ek, \emph{{Many faces
  of symmetric edge polytopes}}, Eprint \url{arXiv:1910.05193}.

\bibitem{MR4081480}
Emanuele Delucchi and Linard Hoessly, \emph{Fundamental polytopes of metric
  trees via parallel connections of matroids}, European J. Combin. \textbf{87}
  (2020), 103098, 18, URL \url{https://doi.org/10.1016/j.ejc.2020.103098}.
  \MR{4081480}

\bibitem{MR3395979}
Akihiro Higashitani, \emph{Smooth {F}ano polytopes arising from finite directed
  graphs}, Kyoto J. Math. \textbf{55} (2015), no.~3, 579--592, URL
  \url{https://doi.org/10.1215/21562261-3089073}. \MR{3395979}

\bibitem{MR3949939}
Akihiro Higashitani, Katharina Jochemko, and Mateusz Micha{\l}ek,
  \emph{Arithmetic aspects of symmetric edge polytopes}, Mathematika
  \textbf{65} (2019), no.~3, 763--784, URL
  \url{https://doi.org/10.1112/s0025579319000147}. \MR{3949939}

\bibitem{MR3810570}
Filip~D. Jevti\'{c}, Marija Jeli\'{c}, and Rade~T. \v{Z}ivaljevi\'{c},
  \emph{Cyclohedron and {K}antorovich-{R}ubinstein polytopes}, Arnold Math. J.
  \textbf{4} (2018), no.~1, 87--112, URL
  \url{https://doi.org/10.1007/s40598-018-0083-4}. \MR{3810570}

\bibitem{MR4249903}
Filip~D. Jevti\'{c}, Marinko Timotijevi\'{c}, and Rade~T. \v{Z}ivaljevi\'{c},
  \emph{Polytopal {B}ier spheres and {K}antorovich-{R}ubinstein polytopes of
  weighted cycles}, Discrete Comput. Geom. \textbf{65} (2021), no.~4,
  1275--1286, URL \url{https://doi.org/10.1007/s00454-019-00151-5}.
  \MR{4249903}

\bibitem{MR2842918}
Tetsushi Matsui, Akihiro Higashitani, Yuuki Nagazawa, Hidefumi Ohsugi, and
  Takayuki Hibi, \emph{Roots of {E}hrhart polynomials arising from graphs}, J.
  Algebraic Combin. \textbf{34} (2011), no.~4, 721--749, URL
  \url{https://doi.org/10.1007/s10801-011-0290-8}. \MR{2842918}

\bibitem{MR3331969}
A.~M. Vershik, \emph{Classification of finite metric spaces and combinatorics
  of convex polytopes}, Arnold Math. J. \textbf{1} (2015), no.~1, 75--81, URL
  \url{https://doi.org/10.1007/s40598-014-0005-z}. \MR{3331969}

\end{thebibliography}
\bibliographystyle{amsplain-url}

\end{document}